\documentclass[12pt]{article}

\usepackage[margin=1in]{geometry}  
\usepackage{graphicx}              
\usepackage{amsmath}               
\usepackage{amsfonts}              
\usepackage{amsthm}                
\usepackage[utf8]{inputenc}
\usepackage[T1]{fontenc}
\usepackage{lmodern} 
\usepackage{latexsym}
\usepackage{hyperref}
\usepackage{color}

\newtheorem{thm}{Theorem}[section]
\newtheorem{lem}[thm]{Lemma}

\newtheorem*{remark}{Remark}

\newcommand\sH{{\mathcal H}}

\newcommand\bZ{{\mathbb Z}}

\newcommand\bN{{\mathbb N}}

\def\Ric{\mathop{\rm Ric}\nolimits}

\def\tr{\mathop{\rm tr}\nolimits}

\def\cD{{\mathcal D}}

\def\bZ{{\mathbb Z}}
\def\bN{{\mathbb N}}

\begin{document}

\title{Deformations from a given K\"ahler metric to a twisted cscK metric}
\date{\today}
\author{Yu Zeng}

\maketitle
\section{Introduction}
In 1950's, E. Calabi(cf. \cite{Ca1},\cite{Ca2}) has raised the famous Calabi's conjecture stating that there exists a unique K\"ahler metric in any given K\"ahler class whose Ricci form is any given 2-form representing the first chern class. This conjecture was later proved in 1970's by the celebrated works of S. T. Yau(\cite{Y}), E. Calabi(\cite{Ca3}) and T. Aubin(\cite{A}) using continuity method to solve the complex Monge-Amp\`ere equation. In particular, when the first chern class is zero or negative, their works imply the existence of K\"ahler-Einstein metrics. When the first chern class is positive, Tian has
made contributions towards understanding precisely when a solution exists(\cite{T1}).

In 1980's, E. Calabi(cf. \cite{Ca4}\cite{Ca5}) proposed a broader program aiming to find the extremal metrics as the generalization of K\"ahler-Einstein metrics in an arbitrary K\"ahler class. As a special case of extremal metrics, the existence problem of the constant scalar curvature K\"ahler(cscK) metrics fits into a general picture of symplectic geometry as described by S. K. Donaldson(\cite{D2}). It was well known by now that  the existence of K\"ahler-Einstein metrics or cscK metrics was equivalent to some notion of "stability" in algebraic geometry (c.f. Yau-Tian-Donaldson conjecture(\cite{RT}, \cite{T1} and \cite{D1})). Recently, this conjecture was settled in the Fano case by the crucial contributions of Chen-Donaldson-Sun (cf. \cite{CDS1}, \cite{CDS2} and \cite{CDS3}). 

In a series of remarkable work (\cite{D3},\cite{D4}, \cite{D5} and \cite{D6}), S. K. Donaldson proved the existence of cscK metric on a K-stable toric surface. Very little was known for a general K\"ahler class in higher dimensions. Recently, X. Chen initiates a new program attacking the existence problem of cscK metrics via a new  continuity path in \cite{C1}, which connects
 the usual cscK metric equation with a second order elliptic equation. As in \cite{C1}, for a positive closed (1,1)-form $\chi$, we define the K\"ahler metric $\omega_{\varphi}$ satisfying 
\begin{align} \label{path}
t(R_{\varphi} - \underline R) - (1-t) (\tr_{\varphi} \chi - \underline \chi) = 0
\end{align}
the twisted cscK metric, where $R_{\varphi}$ denote the scalar curvature of $\omega_{\varphi}$, $\underline R = \frac{[c_1(M)][\omega]^{[n-1]}}{[\omega]^{[n]}}$ and $\underline \chi = \frac{[\chi][\omega]^{[n-1]}}{[\omega]^{[n]}}$. In his same paper, X. Chen also showed the openness of this path when $0< t <1$. And in a subsequent paper \cite{CPZ}, X.Chen, M. P\u aun and Y. Zeng used the openness result at $t=1$ to give a new proof of the uniqueness theorem of extremal metrics. Similar notions of twisted cscK metrics could also be found in earlier papers of J. Fine \cite{JF}, J. Stoppa \cite{JS} and Lejmi-Sz\'ekelyhidi \cite{MS}. 

In this paper, we'll prove the openness of the new continuity path introduced in \cite{C1} at $t=0$. And it adds further evidence that the path is the right one to work on. Following from a simple observation, we could choose $\chi = \omega$ in $(\ref{path})$ such that $(\ref{path})$ always has trivial solution at $t=0$. Our purpose of this paper is to prove the following main theorem:
\begin{thm}\label{thm:main}
Suppose $(M, \omega)$ is a closed K\"ahler manifold. Then, for any $t> 0$ sufficiently small, there exist a unique smooth K\"ahler metric $\omega_{\varphi_t}$ such that
\begin{align}\label{eqn:1.1}
t(R_{\varphi_t} - \underline R) - (1-t) (\tr_{\varphi_t} \omega - n) = 0.
\end{align}
\end{thm}
Notice here, when $t > 0$, $(\ref{eqn:1.1})$ is a 4th order nonlinear elliptic equation while at $t=0$, we get a second order equation. Thus, it's not clear which function space we should choose if we want to apply the inverse function theorem. Fortunately, if we denote $\varphi_0 =0$, then
\begin{align*}
t(R_{\varphi_0} - \underline R) - (1-t) (\tr_{\varphi_0}\omega - n) = t(R_{\varphi_0} - \underline R) \rightarrow 0, \text{ as } t \rightarrow 0
\end{align*}
in any $C^{k}(M)$ norm. It suggests that $\varphi_0$ is very close to a twisted cscK metric when $t> 0$ sufficiently small. Thus, if we take $\varphi_0$ as base point and apply the inverse funcion theorem at $\varphi = \varphi_0$, there's a slight chance that it contains "0" in its neighborhood of image of $\varphi_0$ where every element has a pre-image. However, later we find out that the radius of neighbourhood of $t(R_{\varphi_0} -\underline R)$ which has pre-images decreases faster than $t^2$ while "0" lies in only the radius t neighborhood.

To overcome this difficulty, we'll first introduce basic notions in Section $\ref{sec2}$ and reduce $(\ref{eqn:1.1})$ from a 4th order equation to a second order equation
\begin{align}
r \theta_{\varphi} + \varphi = 0.
\end{align}
Then in Section $\ref{sec3}$ we could expand the above equation in power series of r and collect the same order terms of r to see possible ways of cancelations. Then we could choose $\varphi_1$ closer to the critical point than $\varphi_0$ as shown in Lemma $\ref{lem1}$. Namely, we choose $\varphi_1$ such that "0" is in the $r^4$ neighborhood of $r\theta_{\varphi_1} + \varphi_1$ in $C^\alpha$ space. And in Section $\ref{sec4}$, by intense calculations, we show that the radius of neighborhood of $r\theta_{\varphi_1} + \varphi_1$ which has pre-images is greater than $r^{3+\epsilon}$ for some $\epsilon >0 $. Eventually "0" will fall into the $r^{3+\epsilon}$ neighborhood of image of $\varphi_1$. Thus, it has a pre-image.

Without further notice, the "C" in each estimate means a constant depending on the complex dimension $n$, the background metric $\omega$, the topological constant $\underline{R}$ and $0< \alpha <1$ unless specified.\\

\noindent {\bf Acknowledgement} The author is very grateful to his advisor Prof. X. X. Chen for constant support and encouragement. He also wishes to thank Prof. E. Bedford, Yuanqi Wang and Song Sun for comments on an earlier version of this preprint. During the preparation of this paper, we learned from \cite{C1} that Y. Hashimoto has also announced results similar to Theorem $\ref{thm:main}$.

\section{Preliminary}\label{sec2}

Suppose $(M, \omega)$ is a closed K\"ahler manifold. Denote the space of normalized smooth K\"ahler potentials as
\begin{align}
\sH_{\omega} = \{\varphi \in C^\infty(M) | \omega_{\varphi }  = \omega + \sqrt{-1} \partial \bar \partial \varphi > 0, \int_M \varphi \omega^n = 0\}.
\end{align}
For $\varphi \in \sH_{\omega}$, we denote $R_{\varphi}$ the scalar curvature of $\omega_{\varphi}$ and $\underline R = \frac{[c_1(M)] [\omega]^{[n-1]}}{[\omega]^{[n]}}$.

In \cite{C1}, Chen has introduced a continuity path in $\sH_{\omega}$ for a closed positive (1,1)-form $\chi$ as
\begin{align}\label{eqn:path}
t(R_{\varphi} - \underline{R}) - (1 - t) (\tr_{\varphi} \chi - \underline\chi )= 0, 
\end{align}
where $\underline \chi = \frac{[\chi] [\omega]^{[n-1]}}{[\omega]^{[n]}}$.

In particular, as described in the introduction, we could simply choose the closed positive (1,1)-form to be $\omega$. Thus, $(\ref{eqn:path})$ always has a trivial solution at $t=0$. As in defining the Futaki invariant in \cite{T2}, we could solve the Laplacian equation for any $\varphi \in \sH_{\omega}$
\begin{align}\label{eqn1}
\Delta_{\varphi} f = R_{\varphi} - \underline R.
\end{align}
We denote the solution of $(\ref{eqn1})$ as $\theta_{\varphi}$ with the normalization $\int_M \theta_{\varphi} \omega^n = 0$. Since 
\begin{align}
R_{\varphi} - \underline R &= \tr_{\varphi} (\Ric_{\varphi} - \Ric (\omega) + \Ric (\omega) - \underline R \omega - \underline R  \sqrt{-1} \partial \bar \partial \varphi ),\\
& = \Delta_{\varphi} (- \log \frac{\omega_{\varphi}^n}{\omega^n} - \underline R \varphi) + \tr_{\varphi} (\Ric(\omega) - \underline R \omega),
\end{align}
we have
\begin{align}\label{eqn:def}
\theta_{\varphi} = - \log \frac{\omega_{\varphi}^n}{\omega^n} - \underline R \varphi + \int_M \log \frac{\omega_{\varphi}^n}{\omega^n} \omega^n + P_{\varphi},
 \end{align}
where $P_{\varphi}$ is determined by 
\begin{align}
\Delta_{\varphi} P_{\varphi} = \tr_{\varphi} (\Ric(\omega) - \underline R \omega), \int_M P_{\varphi} \omega^n = 0.
\end{align}
Given the notion of $\theta_{\varphi}$ above, we could reduce the continuity path equation $(\ref{eqn:path})$ with $\chi = \omega$ from a 4th order PDE to a Monge-Amp\`ere type of equation as
\begin{align}
t \theta_{\varphi} + (1-t) \varphi = 0.
\end{align}

Following the discussion above, Theorem $\ref{thm:main}$ will be an easy corollary of the following theorem:
\begin{thm}\label{thm1}
Suppose $(M, \omega)$ is a closed K\"ahler manifold. Then, for any $r> 0$ sufficiently small, there exists a unique $\varphi_r \in \sH_{\omega}$ such that
\begin{align}
r\theta_{\varphi_r} + \varphi_r = 0.
\end{align}
\end{thm}

In our paper, we'll repeatedly use schauder estimate of Laplacian equation. Thus, let's introduce it here as the following Lemma:
\begin{lem}\label{lem2}
If $\varphi \in C^{2,\alpha}(M)$ with $\|\varphi\|_{C^{2,\alpha}(M)} \leq \frac{1}{2}$ and $u\in C^{2,\alpha}(M)$ statisfies
\begin{align}\label{eqn:claim}
\Delta_{\varphi} u = f, \int_M u \omega^n = 0
\end{align}
for some $f \in C^{\alpha}(M)$. Then
\begin{align}
\|u\|_{C^{2,\alpha}(M)} \leq C\|f\|_{C^{\alpha}(M)}.
\end{align}
\end{lem}
\begin{proof} Proof of Lemma $\ref{lem2}$. By Schauder estimate \cite{GT}, we can get
\begin{align}
\|u\|_{C^{2,\alpha}(M)} \leq C\big(\|u\|_{L^{\infty}(M)} + \|f\|_{C^{\alpha}(M)}\big).
\end{align}
To bound $\|u\|_{L^{\infty}(M)}$, we first multiply $u$ on both hand sides of $(\ref{eqn:claim})$, integrate against $\omega_{\varphi}^n$, and we get that
\begin{align}
\int_M |\nabla u|_{\varphi}^2 \omega_{\varphi}^n = \int_M -fu \omega_{\varphi}^n \leq C \|f\|_{L^{\infty}(M)} \|u\|_{L^2(M,\omega)}. 
\end{align}
On the other hand,
\begin{align}
\int_M |\nabla u|_{\varphi}^2 \omega_{\varphi}^n \geq \frac{1}{C} \int_M |\nabla u |^2 \omega^n \geq \frac{1}{C} \|u\|_{L^2(M, \omega)}^2.
\end{align}
Thus, combining the above two inequalities, we can get
\begin{align}
\|u\|_{L^2(M, \omega)} \leq C\|f\|_{L^{\infty}(M)}.
\end{align}
Then, by Moser iteration \cite{GT}, we can get that 
\begin{align}
\|u\|_{L^{\infty}(M)} \leq C\big(\|u\|_{L^2(M)} + \|f\|_{L^{\infty(M)}}\big) \leq C\|f\|_{C^{\alpha}(M)}.
\end{align}
This ends the proof.
\end{proof}
\section{Choose the base point} \label{sec3}

Let's first introduce the space we're going to work on. Define for $0< \alpha <1$ and $k \in \bN$
\begin{align}
\sH^{2,\alpha}_{\omega} &= \{\varphi \in C^{2,\alpha}(M)| \omega_{\varphi} = \omega+ \sqrt{-1} \partial \bar \partial \varphi > 0, \int_M \varphi \omega^n = 0\},\\
C_{\omega}^{k, \alpha}(M) &= \{f \in C^{k, \alpha} (M) | \int_M f \omega^n = 0\}.
\end{align}
More generally, $\theta_{\varphi}$ could be defined on the space $\sH^{2,\alpha}_{\omega}$ if we took the definition as in $(\ref{eqn:def})$. Therefore, we define, still denoted by $\theta_{\varphi}$, 
\begin{align*}
\theta: \sH_{\omega}^{2,\alpha} &\rightarrow C_{\omega}^{\alpha}(M)\\
\varphi &\mapsto \theta_{\varphi} =  - \log \frac{\omega_{\varphi}^n}{\omega^n} - \underline R \varphi + \int_M \log \frac{\omega_{\varphi}^n}{\omega^n} \omega^n + P_{\varphi}, 
\end{align*}
where $P_{\varphi} \in C_{\omega}^{2,\alpha}(M) \subset C_{\omega}^{\alpha}(M)$ is determined by 
\begin{align}
\Delta_{\varphi} P_{\varphi} = \tr_{\varphi} (\Ric(\omega) - \underline R \omega), \int_M P_{\varphi} \omega^n = 0.
\end{align}
Define
\begin{align*}
F_r: \sH_{\omega}^{2,\alpha} &\rightarrow C_{\omega}^{\alpha}(M)\\
  \varphi &\mapsto r \theta_{\varphi} + \varphi.
\end{align*}
Denote $\varphi_0 = 0 \in \sH^{2,\alpha}_{\omega}$. As we described the difficulties in the introduction, $\varphi_0$ is not enough for our purpose. We need to find "better" base point to apply inverse function theorem. 

Let 
\begin{align}
\varphi_1 = \varphi_0 + ru_1 + \frac{r^2}{2} u_2 + \frac{r^3}{6} u_3,
\end{align} 
where $u_i's$ are smooth functions on $M$ with $\int_M u_i \omega^n = 0$ that we'll specify later. First we'll expand $F_r(\varphi_1)$ in terms of $r$ at $r = 0$. Denote $u_r = ru_1 + \frac{r^2}{2} u_2 + \frac{r^3}{6} u_3 $. Compute
\begin{align}
\frac{\partial \theta_{\varphi_1}}{\partial r} = - \Delta_{\varphi_1} \dot u_r - \underline R \dot u_r + \int_M \Delta_{\varphi_1} \dot u_r \omega^n+ \cD P|_{\varphi_1} (\dot u_r),
\end{align}
where $\cD P|_{\varphi_1} : C_{\omega}^{2,\alpha}(M) \rightarrow C_{\omega}^{2,\alpha}(M)$ is the linearization of $P_{\varphi}$ at $\varphi = \varphi_1$ and it satisfies
\begin{align}
\Delta_{\varphi_1} \big(\cD P|_{\varphi_1}(u)\big) = \langle \partial \bar \partial u, \partial \bar \partial P_{\varphi_1} - (\Ric(\omega) - \underline R \omega)\rangle_{\varphi_1}, \int_M \big(\cD P|_{\varphi_1}(u)\big) \omega^n = 0.
\end{align}
Take one more derivative of $\theta_{\varphi_1}$, we get
\begin{align}
\frac{\partial^2 \theta_{\varphi_1}}{\partial^2 r} = - (\Delta_{\varphi_1} \ddot u_r - |\partial \bar \partial \dot u_r|_{\varphi_1}^2 )- \underline R \ddot u_r + \int_M ( \Delta_{\varphi_1} \ddot u_r - |\partial \bar \partial \dot u_r|_{\varphi_1}^2 ) \omega^n + \cD P|_{\varphi_1} (\ddot u_r) + (\frac{\partial}{\partial \varphi} \cD P|_{\varphi})|_{\varphi_1} (\dot u_r, \dot u_r)
\end{align}
where the last term is given by the unique solution of the following elliptic equation
\begin{align*}
\Delta_{\varphi_1} f& = 2 \langle \partial \bar \partial \dot u_r,  \partial \bar \partial \big(\cD P|_{\varphi_1} (\dot u_r) \big) \rangle_{\varphi_1} - \dot u_{r, i\bar p} \dot u_{r, p \bar j} \big( P_{\varphi_1, j \bar i} - (\Ric(\omega) - \underline R \omega)_{j \bar i}\big) \\
&- \dot u_{r, i\bar p} \dot u_{r,  j \bar i} \big( P_{\varphi_1, p \bar j} - (\Ric(\omega) - \underline R \omega)_{p \bar j}\big)
\end{align*}
with $\int_M f \omega^n = 0$. Thus, we get the expansion of $F_r(\varphi_1)$ of $r$ at $r=0$, 
\begin{align}
F_r (\varphi_1)  & = r \theta_{\varphi_1} + \varphi_1 \\
& = \varphi_0 + r (u_1 + \theta_{\varphi_0}) + \frac{r^2}{2} (u_2 + 2 \frac{\partial \theta_{\varphi_1}}{\partial r}|_{r = 0} ) + \frac{r^3}{6} (u_3 + 3 \frac{\partial^2 \theta_{\varphi_1}}{\partial^2 r} |_{r=0})  + O(r^4).
\end{align}
It suggests that we should define
\begin{align*} 
u_1 & = - \theta_{\varphi_0}\\
u_2 &= - 2  \frac{\partial \theta_{\varphi_1}}{\partial r}|_{r = 0} = - 2 \big(- \Delta_{\varphi_0} u_1 - \underline R u_1 + \cD P|_{\varphi_0} (u_1) \big)\\
u_3 & = - 3 \frac{\partial^2 \theta_{\varphi_1}}{\partial^2 r} |_{r=0} = -3\big(- \Delta_{\varphi_0} u_2 - \underline R  u_2 + \cD P|_{\varphi_0} (u_2) + |\partial \bar \partial u_1|_{\varphi_0}^2 - \int_M |\partial \bar \partial u_1|_{\varphi_0}^2 \omega^n\\
& + (\frac{\partial}{\partial \varphi} \cD P|_{\varphi} )|_{\varphi_0} (u_1, u_1)\big)
\end{align*}

It's clear from definitions that $u_i's$ are fixed smooth functions with $C^k$ norm bounds only depend on $\varphi_0$. Therefore, we could choose $r> 0$ sufficiently small such that $\varphi_1 \in \sH_{\omega}^{2,\alpha}$ with $\|\varphi_1\|_{C^{2,\alpha}(M)} \leq \frac{1}{2}$. And we expect that $F_r(\varphi_1)$ is $r^4$ close to "0" in appropriate norms. This observation can be made more precise as the following lemma:
\begin{lem}\label{lem1}
Notations as described above, for $r>0$ sufficiently small, we have 
$$\|F_r(\varphi_1)\|_{C^{\alpha}(M)} \leq Cr^4.$$
\end{lem}
\begin{proof} Proof of Lemma $\ref{lem1}$. It suffices to show that 
\begin{align}
\|\theta_{\varphi_1} - (\theta_{\varphi_0} + r  \frac{\partial \theta_{\varphi_1}}{\partial r}|_{r= 0}  +  \frac{r^2}{2}\frac{\partial^2 \theta_{\varphi_1}}{\partial^2 r}|_{r=0})\|_{C^{\alpha}(M)} \leq C r^3.
\end{align}
By Taylor expansion theorem, we could write the remaining error of the function $\theta_{\varphi_1}$ and its second order taylor expansion as an integral,
\begin{align}
R(x) = \frac{1}{2!}\int_0^r (r-s)^2\big(\frac{\partial^3 \theta_{\varphi_1}}{\partial^3 r}|_{r = s} (x)\big) ds.
\end{align}
So it suffices to show that for any $s \in [0,r]$ with $r > 0$ sufficiently small
\begin{align}
\|  \frac{\partial^3 \theta_{\varphi_1}}{\partial^3 r}|_{r = s}\|_{C^{\alpha}(M)} \leq C.
\end{align}
Denote $\varphi_s = \varphi_0 + su_1 + \frac{s^2}{2} u_2 + \frac{s^3}{6} u_3$ and $u_s = su_1 + \frac{s^2}{2} u_2 + \frac{s^3}{6} u_3 $. Compute 
\begin{align*}
\frac{\partial^3 \theta_{\varphi_1}}{\partial^3 r}|_{r = s} & = - (\Delta_{\varphi_s} u_s^{(3)}- 3 \langle \partial \bar \partial u_s^{(1)}, \partial \bar \partial u_s^{(2)}\rangle_{\varphi_s} + 2 (\partial \bar \partial u_s^{(1)})^{*3})\\
& - \int_M(\Delta_{\varphi_s} u_s^{(3)}- 3 \langle \partial \bar \partial u_s^{(1)}, \partial \bar \partial u_s^{(2)}\rangle_{\varphi_s} + 2 (\partial \bar \partial u_s^{(1)})^{*3}) \omega^n - \underline R u_s^{(3)}\\
&+ \cD P|_{\varphi_s} (u_s^{(3)}) + 2 (\frac{\partial}{\partial \varphi} \cD P|_{\varphi})|_{\varphi = \varphi_s} (u_s^{(2)}, u_s^{(1)}) +  (\frac{\partial}{\partial \varphi} \cD P|_{\varphi})|_{\varphi = \varphi_s} (u_s^{(1)}, u_s^{(2)}) \\
& + (\frac{\partial^2}{\partial^2 \varphi} \cD P|_{\varphi})|_{\varphi = \varphi_s} (u_s^{(1)},u_s^{(1)},u_s^{(1)}).
\end{align*}
It's obvious that the first two lines has uniform $C^{\alpha}$ norm as we expected. Therefore, we need to estimate the last four terms of the above equation. Let's first consider $P_{\varphi_s}$. It satisfies the  Laplacian equation as described in Lemma $\ref{lem2}$, so we get that
\begin{align}\label{eqn:3.1}
\|P_{\varphi_s}\|_{C^{2,\alpha}(M)} \leq C.
\end{align}
Then we can estimate $\cD P|_{\varphi_s} (u)$ using $(\ref{eqn:3.1})$ and Lemma $\ref{lem2}$ since it satisfies the similar Laplacian equation with right hand side depending on second order derivatives of $P_{\varphi_s}$, we can conclude that
\begin{align}
\|\cD P |_{\varphi_s} (u)\|_{C^{2,\alpha}(M)} \leq C \|u\|_{C^{2,\alpha}(M)}.
\end{align}
Thus, we could further estimate the term using the same argument in Lemma $\ref{lem2}$
\begin{align}
\|(\frac{\partial }{\partial \varphi} \cD P|_{\varphi})|_{\varphi_s} (u, v)  \|_{C^{2,\alpha}(M)} \leq C\|u\|_{C^{2,\alpha}(M)} \|v\|_{C^{2,\alpha}(M)}.
\end{align}
Finally, we could estimate the term $(\frac{\partial^2}{\partial^2 \varphi} \cD P|_{\varphi})|_{\varphi_s} (u, v, w)$ which satisfies the equation
\begin{align}
\Delta_{\varphi_s} f &= \langle \partial \bar \partial w, \partial \bar \partial \big( (\frac{\partial }{\partial \varphi} \cD P|_{\varphi})|_{\varphi_s} (u, v)  \big) \rangle_{\varphi_s} + \langle \partial \bar \partial u, \partial \bar \partial \big( (\frac{\partial }{\partial \varphi} \cD P|_{\varphi})|_{\varphi_s} (v,w)  \big) \rangle_{\varphi_s} \\
&+ \langle \partial \bar \partial v, \partial \bar \partial \big( (\frac{\partial }{\partial \varphi} \cD P|_{\varphi})|_{\varphi_s} (u, w)  \big) \rangle_{\varphi_s} + \partial \bar \partial v * \partial \bar \partial w * \partial \bar \partial \big(\cD P|_{\varphi_s}(u) \big)\\
& + \partial \bar \partial u * \partial \bar \partial w * \partial \bar \partial \big(\cD P|_{\varphi_s}(v) \big) + \partial \bar \partial v * \partial \bar \partial u * \partial \bar \partial \big(\cD P|_{\varphi_s}(w) \big)\\
& + \partial \bar \partial u * \partial \bar \partial v * \partial \bar \partial w * \big(\partial \bar \partial P_{\varphi_s} - \Ric(\omega) - \underline{R} \omega \big), \int_M f \omega^n = 0.
\end{align}
Thus by the Lemma $\ref{lem2}$, we can conclude that
\begin{align}
\|(\frac{\partial^2}{\partial^2 \varphi} \cD P|_{\varphi})|_{\varphi_s} (u, v, w)\| \leq C\|u\|_{C^{2,\alpha}(M)}\|v\|_{C^{2,\alpha}(M)}\|w\|_{C^{2,\alpha}(M)}.
\end{align}
Since $\|u_s^{(i)}\|_{C^{2,\alpha}(M)} \leq C$ for $1 \leq i \leq 3$,
\begin{align}
 \|\frac{\partial^3 \theta_{\varphi_1}}{\partial^3 r}|_{r = s}\|_{C^{\alpha}(M)} \leq C. 
\end{align}
Thus it ends the proof of the lemma.
\end{proof}

\section{Proof of Theorem $\ref{thm:main}$} \label{sec4}
In last section, we have shown that for $r> 0$ sufficiently small, $\|F_r(\varphi_1)\|_{C^{\alpha}(M)} \leq Cr^4$. Next we'll construct a contract map defined on a $r^{1+ \epsilon}$ neighborhood of $\varphi_1$ in $C^{2,\alpha}$ space, which is similar to the proof of inverse function theorem. Since $F_r(\varphi_1)$ is $r^4$ small in $C^{\alpha}$ norm, we could start the iterating process from $\varphi_1$ and keep every following term stay within the precribed $r^{1+\epsilon}$ neighborhood of $\varphi_1$.

First, we have to understand the linearization of $F_r: \sH^{2,\alpha}_{\omega} \rightarrow C^{\alpha}_{\omega}(M)$ at $\varphi =\varphi_1$. Compute
\begin{align*}
\cD F_r|_{\varphi_1} : C_{\omega}^{2,\alpha} (M) &\rightarrow C_{\omega}^{\alpha}(M)\\
& u \mapsto - r\Delta_{\varphi_1} u + (1- r \underline R) u + r \big(\int_M (\Delta_{\varphi_1} u )\omega^n + \cD P|_{\varphi_1} (u) \big),
\end{align*}
where $\cD P|_{\varphi_1} (u)$ satisfies
\begin{align}\label{eqn7}
\Delta_{\varphi_1} \big(\cD P|_{\varphi_1} (u)\big) = \langle \partial \bar \partial u, \big(\partial \bar \partial P_{\varphi_1} - (\Ric(\omega) - \underline R \omega)\big)\rangle_{\varphi_1}, \int_M \big(\cD P|_{\varphi_1} (u) \big) \omega^n = 0.
\end{align}

We summarize the properties of $\cD F_r|_{\varphi_1}$ as the following lemma:
\begin{lem}\label{lem3}
Suppose $0< \alpha <1$. Then, for $r> 0$ sufficiently small, the linearizaiton of $F_r: \sH^{2,\alpha}_{\omega} \rightarrow C^{\alpha}_{\omega}(M)$ at $\varphi= \varphi_1$, $\cD F_r|_{\varphi_1}: C_{\omega}^{2,\alpha} (M) \rightarrow C_{\omega}^{\alpha}(M)$, is injective and also surjective. Moreover, the operator norm of the inverse of $\big(\cD F_r|_{\varphi_1}\big)$ has the upper bound
$$\|\big( \cD F_r |_{\varphi_1}\big)^{-1}\|\leq C r^{- \frac{2-\alpha}{1-\alpha}}.$$
\end{lem}

Before proving Lemma $\ref{lem3}$, we'll need the estimate of $\cD P|_{\varphi}(u)$ for $\|\varphi\|_{C^{2,\alpha}(M)} \leq \frac{1}{2}$. We summarize it as the following lemma:
\begin{lem}\label{lem:Lp}
Suppose $\|\varphi\|_{C^{2,\alpha}(M)} \leq \frac{1}{2}$, then we have the estimate for any $1 < p < \infty$, 
\begin{align}
\|\big(\cD P|_{\varphi} (u)\big) \|_{L^{p}(M)} \leq C_p \| u\|_{L^{p}(M)}.
\end{align}
\end{lem}
\begin{remark}
Since $\omega$ and $\omega_{\varphi}$ are equivalent metrics if $\|\varphi\|_{C^{2,\alpha}(M)}\leq \frac{1}{2}$, we make no efforts to distinguish between $L^p$ spaces with respect to the two metrics hereafter.
\end{remark}
\begin{proof}

We first introduce the Green function $G_{\varphi}(x, y)$ of the metric $\omega_{\varphi}$. Then we define 
\begin{align}
T (u) (x)& = \int_M G_{\varphi}(x , y) \big(u(P_{\varphi, i\bar j} - (\Ric(\omega) - \underline R \omega)_{i\bar j} ) \big)_{, \bar i  j}(y) \omega_{\varphi}^n \\
& = \int_M (G_{\varphi}(x,y))_{, \bar i j} \big( u (P_{\varphi, i\bar j} - (\Ric(\omega) - \underline R \omega)_{i\bar j} )\big) (y) \omega_{\varphi}^n.
\end{align}
Since 
\begin{align}
\Delta_{\varphi} \big(\cD P|_{\varphi} (u)\big) =\big(u(P_{\varphi, i\bar j} - (\Ric(\omega) - \underline R \omega)_{i\bar j} ) \big)_{, \bar i  j} , \int_M \big(\cD P|_{\varphi} (u) \big) \omega^n = 0, 
\end{align}
we have 
\begin{align}\label{eqn6}
\cD P|_{\varphi} (u) = T(u) - \int_M T(u )\omega^n
\end{align}
For $i, j  \in \bN$, we define the operator $$T_{\bar i j} f = \int_M (G_{\varphi}(x,y))_{, \bar i j} f (y) \omega_{\varphi}^n.$$

$T_{\bar i j}$ is a Calderon-Zygmund(\cite{GT}) operator which maps $L^p$ functions to $L^p $ functions for any $1 < p < \infty$. Moreover we can show $T_{\bar i j}$ has uniform norms. To see this, we consider the Laplacian equation
\begin{align}\label{eqn5}
\Delta_{\varphi} u = f, \int_M u \omega_{\varphi}^n = 0.
\end{align} 
Thus, we see that the solution satisfies
\begin{align}
\frac{\partial^2}{\partial z_{\bar i}\partial z_j } u (x)= \big(T_{\bar i j} f \big)(x).
\end{align}
So it suffices to show the uniform $W^{2,2}$ estimates of $(\ref{eqn5})$, which follows from the fact that $\|\varphi\|_{C^{2,\alpha}(M)} \leq \frac{1}{2}$ and the standard $L^p$ theory of elliptic equation(\cite{GT}). We have the estimate for any $p\in  (1,+ \infty)$
\begin{align}
\|T_{\bar i j} f\|_{L^p(M)} \leq C_p \|f\|_{L^p(M)}.
\end{align}
Thus, taking advantages of the above estimate, we can get
\begin{align}
\|T(u)\|_{L^{p}(M)} & \leq \sum_{k,l} C_p \|u\big( g_{\varphi}^{ i \bar l} g_{\varphi}^{k \bar j} (P_{\varphi, i\bar j} - (\Ric(\omega) - \underline R \omega)_{i\bar j})\big)\|_{L^p(M)}\\
&\leq C_p \|u\|_{L^p(M)}.
\end{align}
Thus, we have for any $1 <p < \infty$
\begin{align}
\|\big(\cD P|_{\varphi} (u)\big) \|_{L^{p}(M)} \leq C_p \| u\|_{L^{p}(M)}.
\end{align}
This ends the proof of Lemma $\ref{lem:Lp}$.
\end{proof}

Now we can prove Lemma $\ref{lem3}$.

\begin{proof}Proof of Lemma $\ref{lem3}$. First we show that $\cD F_r|_{\varphi_1}$ is injective. Suppose there exists $ u \in C_{\omega}^{2,\alpha}(M)$ such that
\begin{align} \label{eqn3}
- r\Delta_{\varphi_1} u + (1- r \underline R) u + r \big(\int_M (\Delta_{\varphi_1} u )\omega^n + \cD P|_{\varphi_1} (u) \big) = 0.
\end{align}
It suffices to show that $u = 0$. Multiply $u$ on both hand sides of $(\ref{eqn3})$ and integrate against $\omega_{\varphi_1}^n$. 
\begin{align}
0 &= r\int_M |\nabla u|_{\varphi_1}^2 \omega_{\varphi_1}^n + (1- r \underline R) \int_M u^2 \omega_{\varphi_1}^n + r \big(\int_M (\Delta_{\varphi_1} u )\omega^n \big) \big( \int_M u \omega_{\varphi_1}^n\big) + r \int_M \big(\cD P|_{\varphi_1} (u)\big) u \omega_{\varphi_1}^n\\
&\geq (1- r \underline R) \int_M u^2 \omega_{\varphi_1}^n + r \{ \big(\int_M (\Delta_{\varphi_1} u )\omega^n \big) \big( \int_M u \omega_{\varphi_1}^n\big)  +  \int_M \big(\cD P|_{\varphi_1} (u)\big) u \omega_{\varphi_1}^n\}. \label{eqn4}
\end{align}
We focus on estimates of the later two terms in $(\ref{eqn4})$. Consider 
\begin{align*}
\int_M (\Delta_{\varphi_1} u) \omega^n &= \int_M (\Delta_{\varphi_1}u) (\frac{\omega^n}{\omega_{\varphi_1}^n} )\omega_{\varphi_1}^n = \int_M u (\Delta_{\varphi_1} \frac{\omega^n}{\omega_{\varphi_1}^n})\omega_{\varphi_1}^n \\
& = \int_M u \frac{\omega^n}{\omega_{\varphi_1}^n}g_{\varphi_1}^{i\bar j} \big( - g_{\varphi_1}^{k \bar l} \varphi_{1, k\bar l i \bar j} + g_{\varphi_1}^{k\bar p} g_{\varphi_1}^{q\bar l} \varphi_{1, \bar p q \bar j} \varphi_{1, k \bar l i} + g_{\varphi_1}^{\bar p q} g_{\varphi_1}^{k\bar l}\varphi_{1, \bar p q \bar j}  \varphi_{1, k\bar l i}\big)  \omega_{\varphi_1}^n\\
&\geq - Cr \big(\int_M |u| \omega_{\varphi_1}^n\big)
\end{align*}
where the derivatives are covariant derivatives of $\omega$. Thus, we have that
\begin{align}
r  \big(\int_M (\Delta_{\varphi_1} u )\omega^n \big) \big( \int_M u \omega_{\varphi_1}^n\big) \geq -Cr^2\int_M u^2 \omega_{\varphi_1}^n
\end{align}
To estimate the last term of $(\ref{eqn4})$, we need the following estimate of $(\cD P|_{\varphi_1})$. 

Choosing $r>0$ sufficiently small, we can get $\|\varphi_1\|_{C^{2,\alpha}(M)} \leq \frac{1}{2}$. Using Lemma $\ref{lem:Lp}$, we get 
\begin{align}
\|\big(\cD P|_{\varphi_1}(u)\big)\|_{L^2(M)} \leq C \|u\|_{L^2(M)}. 
\end{align}
Thus for the last term in $(\ref{eqn4})$ we have the estimate 
\begin{align}
\int_M \big(\cD P|_{\varphi_1} (u)\big) u \omega_{\varphi_1}^n \geq  - C\int_M u^2 \omega_{\varphi_1}^n.
\end{align}
Therefore, combining the estimates above, we have that 
\begin{align}
0 \geq (1 - Cr) \int_M u^2 \omega_{\varphi_1}^n.
\end{align}
It implies that when $r >0$ sufficiently small, we have that $u = 0$. So we have proved the injectivity of $\big(\cD F_r |_{\varphi_1}\big)$.

Next, we show the surjectivity of $\big(\cD F_r |_{\varphi_1}\big)$ and the upper bound of $\|\big(\cD F_r |_{\varphi_1}\big)^{-1}\|$ together. For $f \in C^{\alpha}_{\omega}(M)$, we'll use continuity method to solve the equation 
\begin{align}
\cD F_r|_{\varphi_1} (u) = f.
\end{align}
Define for $s \in [0,1]$,
\begin{align}
L_s: C^{2,\alpha}(M) &\rightarrow C^{\alpha}(M)\\
u & \mapsto - r\Delta_{\varphi_1} u + (1- r \underline R) u + sr \big( \int_M (\Delta_{\varphi_1} u) \omega^n  +  \cD P|_{\varphi_1} (u) \big).
\end{align}
First, we show that for any $s \in [0,1]$
\begin{align}\label{eqn:continuity}
\|u\|_{C^{2,\alpha}(M)} \leq C_r \|L_s u\|_{C^{\alpha}(M)}.
\end{align}
From the definition of $L_s$, we get that, 
\begin{align}\label{eqn10}
\Delta_{\varphi_1} u = - \frac{1}{r}L_s u + \frac{1- r\underline R}{r} u + s\big( \int_M (\Delta_{\varphi_1} u) \omega^n  +  \cD P|_{\varphi_1} (u) \big).
\end{align}
Since we choose $r>0$ sufficiently small s.t. $\|\varphi_1\|_{C^{2,\alpha}(M)} \leq \frac{1}{2}$, we can get from Schauder estimate,
\begin{align*}
\|u\|_{C^{2,\alpha}(M)} & \leq C\big(\|\Delta_{\varphi_1} u\|_{C^{\alpha}(M)} + \|u\|_{L^{\infty}(M)}\big)\\
 & \leq C\big(\frac{1}{r}\|L_s u\|_{C^{\alpha}(M)} + \frac{1}{r} \|u\|_{C^{\alpha}(M)} + |\int_M (\Delta_{\varphi_1} u )\omega^n| +\| \big(\cD P|_{\varphi_1}(u) \big) \|_{C^{\alpha}(M)} + \|u\|_{L^{\infty}(M)}\big)\\
& \leq C_0\big(\frac{1}{r}\|L_s u\|_{C^{\alpha}(M)} +  \frac{1}{r} \|u\|_{C^{\alpha}(M)} + \| \big(\cD P|_{\varphi_1}(u) \big) \|_{C^{\alpha}(M)} + \|u\|_{L^{\infty}(M)} \big).
\end{align*}
By interpolations \cite{GT}, we have
\begin{align}\label{eqn8}
\|u\|_{C^{\alpha}(M)} \leq \frac{r}{4C_0} \|u\|_{C^{2,\alpha}(M)} + C r^{- \frac{\alpha}{1-\alpha}} \|u\|_{L^{\infty}(M)}.
\end{align}
Also, for term $\|\cD P|_{\varphi_1}(u)\|_{C^{\alpha}(M)}$, since it satisfies equation $(\ref{eqn7})$, we have estimate
\begin{align}\label{eqn9}
\|\cD P|_{\varphi_1} (u)\|_{C^{\alpha}(M)} &\leq C \|\langle \partial \bar \partial u, \big(\partial \bar \partial P_{\varphi_1} - (\Ric(\omega) - \underline R \omega)\big)\rangle_{\varphi_1}\|_{L^{\infty}(M)}\\
& \leq C\|\partial \bar \partial u\|_{L^{\infty}(M)} \leq \frac{1}{4C_0} \|u\|_{C^{2,\alpha}(M)} + C\|u\|_{L^{\infty}(M)}.
\end{align}
Combining estimates of $(\ref{eqn8})$ and $(\ref{eqn9})$, we have that
\begin{align}
\|u\|_{C^{2,\alpha}(M)} \leq C\big( \frac{1}{r}\|L_s u\|_{C^{\alpha}(M)} + r^{-\frac{1}{1-\alpha}} \|u\|_{L^{\infty}(M)}\big).
\end{align}
Now we focus on estimates of $\|u\|_{L^{\infty}(M)}$. For $p>1$, We could first multiply $|u|^{p}$ on both hand sides of $(\ref{eqn10})$ and integrate against $\omega_{\varphi_1}^n$ on the region $\{u> 0\}$. Then we'll get by a similar argument which we use to prove the injectivity,
\begin{align*}
\frac{1}{r}\int_{u>0} (L_s u) u^p \omega_{\varphi_1}^n &\geq \int_{u> 0} p u^{p-1}|\nabla u|_{\varphi_1}^2 \omega_{\varphi_1} + \frac{1-r\underline R}{r}  \int_{u>0} u^{p+1}\omega_{\varphi_1}^n - Cr \big(\int_M |u| \omega_{\varphi_1}^n \big) \big(\int_{u>0} u^p \omega_{\varphi_1}^n \big) \\
& - C_p \|u\|_{L^{p+1}(M)} (\int_{u>0} u^{p+1} \omega_{\varphi_1})^{\frac{p}{p+1}}\\
& \geq \frac{1- r \underline R}{r} \int_{u>0} u^{p+1} \omega_{\varphi_1}^n- C_p \int_{M} u^{p+1} \omega_{\varphi_1}^n.
\end{align*}
Multiply $|u|^{p}$ on both hand sides of $(\ref{eqn10})$ and integrate against $\omega_{\varphi_1}^n$ on the region $\{u<  0\}$. Similarly we get
\begin{align}
- \frac{1}{r}\int_{u< 0} (L_s u) |u|^p \omega_{\varphi_1}^n 
& \geq \frac{1- r \underline R}{r} \int_{u<0} |u|^{p+1} \omega_{\varphi_1}^n- C_p \int_{M} |u|^{p+1} \omega_{\varphi_1}^n.
\end{align}
Thus, we get for $p < p_0 < \infty$, we could choose our $r>0$ small such that 
\begin{align}
\frac{1}{r} \int_M |u|^{p+1} \omega^n \leq \frac{1}{r} \|u\|_{L^{p+1}(M)}^{\frac{p}{p+1}} \|L_s u\|_{L^{p+1}(M)}.
\end{align}
And then for $ p< p_0 +1$
\begin{align}
\|u\|_{L^{p}(M)} \leq C  \|L_s u\|_{L^{p}(M)} \leq C \|L_s u\|_{L^{\infty}(M)}.
\end{align}
By $L^p$ theory of elliptic equation for $(\ref{eqn10})$, we get for $ p < p_0+1$
\begin{align}
\|u\|_{W^{2,p}(M)} &\leq C \big( \frac{1}{r}\|L_su\|_{L^{p}(M)} + \frac{1}{r}\|u\|_{L^{p}(M)} \big)\\
&\leq \frac{C}{r}\|L_s u\|_{L^{\infty}(M)}.
\end{align}
By sobolev embedding, we can get that for p >n
\begin{align}
\|u\|_{L^{\infty}(M)} \leq C \|u\|_{W^{2,p}(M)} \leq \frac{C}{r}\|L_s u\|_{L^\infty(M)}.
\end{align}

Therefore, we conclude that 
\begin{align}
\|u\|_{C^{2,\alpha}(M)} \leq C r^{- \frac{2-\alpha}{1- \alpha}} \|L_s u\|_{C^{\alpha}(M)}
\end{align}
Since the norm is independent of $s \in [0,1]$ and obviously $L_0: C^{2,\alpha}(M) \rightarrow C^{\alpha}(M)$ is onto, thus by continuity method in \cite{GT}, we conclude that $L_1: C^{2,\alpha}(M) \rightarrow C^{\alpha}(M)$ is also onto. Thus we have shown that $\cD F_r|_{\varphi_1} =  L_1$ is surjective. And 
\begin{align}
\|\big(\cD F_r |_{\varphi_1} \big)^{-1} (f)\|_{C^{2, \alpha}(M)} \leq C r^{- \frac{2-\alpha }{1 - \alpha}} \|f\|_{C^{\alpha}(M)}.
\end{align}
This ends the proof of Lemma $\ref{lem3}$.
\end{proof}

Define functional $\Psi$ in a $C^{2,\alpha}$-neighborhood of $\varphi_1$ as 
\begin{align*}
\Psi: \sH^{2,\alpha}_{\omega}& \rightarrow C^{2, \alpha}_{\omega}(M)\\
 \varphi & \mapsto \varphi + \big(\cD F_r|_{\varphi_1}\big)^{-1} \big(-F_r(\varphi)\big)
\end{align*}
Our goal is to find $\varphi \in \sH_{\omega}^{2, \alpha}$ such that $F_r(\varphi) = 0$. Given the definition of $\Psi$, our problem comes down to find the fixed point of $\Psi$. So we need to show that $\Psi$ is a contraction in a small neighborhood of $\varphi_1 \in \sH_{\omega}^{2,\alpha}$.
\begin{lem}\label{lem:contraction}
There exists some $\delta > 0$,  such that if $\varphi, \tilde\varphi \in \sH_{\omega}^{2,\alpha}$ with $\|\varphi - \varphi_1\|_{C^{2,\alpha}(M)} < r^{\frac{1}{1 - \alpha }}\delta$ and $\|\tilde \varphi - \varphi_1\|_{C^{2,\alpha}(M)}< r^{\frac{1}{1 - \alpha }} \delta$, then
\begin{align}
\|\Psi (\varphi) - \Psi (\tilde \varphi) \|_{C^{2, \alpha}(M)} \leq \frac{1}{2} \|\varphi - \tilde \varphi\|_{C^{2,\alpha}(M)}.
\end{align}
\end{lem}

\begin{proof}
Denote $\varphi_s = s\varphi + (1-s) \tilde\varphi$. Suppose $\|\varphi - \varphi_1\|_{C^{2,\alpha}(M)} < r^{\frac{1}{1 - \alpha }}\delta$ and $\|\tilde \varphi - \varphi_1\|_{C^{2,\alpha}(M)}< r^{\frac{1}{1 - \alpha }} \delta$. We'll specify $\delta> 0$ later.

We have
\begin{align*}
\Psi(\varphi) - \Psi(\tilde\varphi) & = \int_0^1 \frac{\partial }{\partial s} \Psi(\varphi_s) ds\\
& = (\varphi - \tilde \varphi) - \int_0^1 \big(\cD F_r|_{\varphi_1}\big)^{-1} \big(\cD F_r|_{\varphi_s} (\varphi - \tilde \varphi)\big) ds \\
& =  -\int_0^1 \big(\cD F_r|_{\varphi_1}\big)^{-1} \{\big(\cD F_r|_{\varphi_s} - \cD F_r|_{\varphi_1} \big)(\varphi - \tilde \varphi) \} ds. 
\end{align*}
We consider the term 
\begin{align*}
\big(\cD F_r|_{\varphi_s} - \cD F_r|_{\varphi_1} \big)(\varphi - \tilde \varphi) & =  r \big(  - (\Delta_{\varphi_s} - \Delta_{\varphi_1})(\varphi -\tilde \varphi) + \int_M \big( (\Delta_{\varphi_s} - \Delta_{\varphi_1})(\varphi -\tilde \varphi)\big) \omega^n \\
&+ (\cD P|_{\varphi_s} - \cD P|_{\varphi_1}) (\varphi -\tilde \varphi)\big).
\end{align*}
Thus, we know that
\begin{align}
& \|\big(\cD F_r|_{\varphi_s} - \cD F_r|_{\varphi_1} \big)(\varphi - \tilde \varphi)\|_{C^{\alpha}(M)}\\
& \leq Cr \big( r^{\frac{1}{1-\alpha}}\delta \|\varphi - \tilde \varphi\|_{C^{2,\alpha}(M)} + \| (\cD P|_{\varphi_s} - \cD P|_{\varphi_1}) (\varphi -\tilde \varphi)\|_{C^{\alpha}(M)}\big).
\end{align}
By definitons of $\cD P|_{\varphi}$ in $(\ref{eqn7})$, we have
\begin{align}
\Delta_{\varphi_1} \big(\cD P|_{\varphi_1} (u)\big) &= \langle \partial \bar \partial u, \big(\partial \bar \partial P_{\varphi_1} - (\Ric(\omega) - \underline R \omega)\big)\rangle_{\varphi_1}, \int_M \big(\cD P|_{\varphi_1} (u) \big) \omega^n = 0\\
\Delta_{\varphi_s} \big(\cD P|_{\varphi_s} (u)\big) & = \langle \partial \bar \partial u, \big(\partial \bar \partial P_{\varphi_s} - (\Ric(\omega) - \underline R \omega)\big)\rangle_{\varphi_s}, \int_M \big(\cD P|_{\varphi_s} (u) \big) \omega^n = 0.
\end{align}
So
\begin{align*}
\Delta_{\varphi_1} \big(\cD P|_{\varphi_1}(u) - \cD P|_{\varphi_s}(u)\big) & =  \langle \partial \bar \partial u, \big(\partial \bar \partial P_{\varphi_1} - (\Ric(\omega) - \underline R \omega)\big)\rangle_{\varphi_1} - \langle \partial \bar \partial u, \big(\partial \bar \partial P_{\varphi_s} - (\Ric(\omega) - \underline R \omega)\big)\rangle_{\varphi_s} \\
&+ \big(\Delta_{\varphi_s} - \Delta_{\varphi_1}\big)\big( \cD P|_{\varphi_s} (u) \big)\\
& = u_{, i \bar j} (g_{\varphi_1}^{i \bar l} g_{\varphi_1}^{k \bar j}  -g_{\varphi_s}^{i \bar l} g_{\varphi_s}^{k \bar j}) P_{\varphi_1, k \bar l}  + u_{,i\bar j} g_{\varphi_s}^{i\bar l} g_{\varphi_s}^{k\bar j}(P_{\varphi_1} - P_{\varphi_s}) \\
&- u_{, i \bar j} (g_{\varphi_1}^{i \bar l} g_{\varphi_1}^{k \bar j}  -g_{\varphi_s}^{i \bar l} g_{\varphi_s}^{k \bar j} ) (\Ric(\omega) - \underline R \omega)_{k\bar l} 
 + (g_{\varphi_s}^{k\bar l} - g_{\varphi_1}^{k\bar l}) \big(\cD P|_{\varphi_s} (u)\big)_{, k \bar l}
\end{align*}
Thus, by schauder estimate and previous estimate about $P_{\varphi}$ and $\cD P|_{\varphi}(u)$ in Section $\ref{sec3}$,
\begin{align*}
\|\cD P|_{\varphi_1}(u) - \cD P|_{\varphi_s}(u)\|_{C^{2,\alpha}(M)} &\leq C r^{\frac{1}{1-\alpha}} \delta \big( \|u\|_{C^{2,\alpha}(M)} +\|\cD P|_{\varphi_s} (u)\|_{C^{2,\alpha}(M)}\big) \\
& + C \|P_{\varphi_1} - P_{\varphi_s}\|_{C^{2,\alpha}(M)}\|u\|_{C^{2,\alpha}(M)}\\
&\leq Cr^{\frac{1}{1-\alpha}}\delta \|u\|_{C^{2,\alpha}(M)}+ C \|P_{\varphi_1} - P_{\varphi_s}\|_{C^{2,\alpha}(M)}\|u\|_{C^{2,\alpha}(M)}.
\end{align*}
Since we have
\begin{align*}
\Delta_{\varphi_1} (P_{\varphi_1} - P_{\varphi_s}) = (g_{\varphi_s}^{k\bar l} - g_{\varphi_1}^{k\bar l})  P_{\varphi_s, k \bar l} + (g_{\varphi_1}^{k\bar l} - g_{\varphi_s}^{k\bar l})(\Ric (\omega) -\underline R\omega)_{k \bar l},
\end{align*}
then
\begin{align}
\|P_{\varphi_1} - P_{\varphi_s}\|_{C^{2,\alpha}(M)} \leq C r^{\frac{1}{1-\alpha}} \delta.
\end{align}
Thus, we have
\begin{align}
\|\big(\cD P|_{\varphi_1} - \cD P|_{\varphi_s}\big) (\varphi -\tilde \varphi)\|_{C^{\alpha}(M)} \leq \|\big(\cD P|_{\varphi_1} - \cD P|_{\varphi_s}\big) (\varphi -\tilde \varphi)\|_{C^{2,\alpha}(M)} \leq C r^{\frac{1}{1- \alpha}} \delta \|\varphi -\tilde \varphi\|_{C^{2,\alpha}(M)}.
\end{align}
By Lemma $\ref{lem3}$, we have that
\begin{align}
\|\big( \cD F_r|_{\varphi_1}\big)^{-1} \{\big(\cD F_r|_{\varphi_s} - \cD F_r|_{\varphi_1} \big)(\varphi - \tilde \varphi) \} \|_{C^{2,\alpha}(M)} &\leq C r^{-\frac{2-\alpha}{1-\alpha}} r  r^{\frac{1}{1-\alpha}} \delta \|\varphi -\tilde \varphi\|_{C^{2,\alpha}(M)}\\
& \leq C \delta \|\varphi - \tilde \varphi\|_{C^{2,\alpha}(M)}.
\end{align}
And then
\begin{align}
\|\Psi(\varphi) - \Psi(\tilde \varphi)\|_{C^{2,\alpha}(M)} \leq C\delta \|\varphi -\tilde \varphi\|_{C^{2,\alpha}(M)}.
\end{align}
We could choose $\delta> 0$ sufficiently small such that $C\delta < \frac{1}{2}$, and thus it ends the proof of Lemma $\ref{lem:contraction}$. 
\end{proof}
Now we're ready to prove the Theorem $\ref{thm1}$.

\begin{proof} Denote the constant $\delta> 0 $ in Lemma $\ref{lem:contraction}$ as $\delta_{0}$. Define for $k \in \bZ$
\begin{align}
\varphi_k = \Psi^{k-1} (\varphi_1).
\end{align}
Ultimately, we want to show that $\varphi_k \rightarrow \varphi_{\infty}$ in $C^{2,\alpha}(M)$ norm for some $\varphi_{\infty} \in \sH_{\omega}^{2,\alpha}$ as $k \rightarrow \infty$. We choose the start point to be $\varphi_1$, thus we need to show that $\varphi_2$ stays in the neighborhood of $\varphi_1$ for $\Psi$ to be contraction. Compute
\begin{align*}
\|\varphi_2 - \varphi_1\|_{C^{2,\alpha}(M)} & = \|\big(\cD F_r|_{\varphi_1}\big)^{-1} \big(-F_r(\varphi_1)\big)\|_{C^{2,\alpha}(M)}\\
& \leq C r^{- \frac{2-\alpha}{1-\alpha}} \|F_r(\varphi_1)\|_{C^{\alpha}(M)} \\
& \leq(C r^{\frac{1-3 \alpha}{1-\alpha}}) r^{\frac{1}{1-\alpha}} .
\end{align*}
where we use Lemma $\ref{lem1}$ and Lemma $\ref{lem3}$. It's obvious we could choose $\alpha =\frac{1}{4}$ and $r>0$ sufficiently small such that 
\begin{align}
\|\varphi_2 - \varphi_1\|_{C^{2,\alpha}(M)} \leq \frac{1}{2} r^{\frac{1}{1-\alpha}} \delta_0.
\end{align}
By induction, we could get that for any $k \in \bZ$
\begin{align}
\|\varphi_k -\varphi_1\|_{C^{2,\alpha}(M)} < r^{\frac{1}{1-\alpha}} \delta_0, 
\end{align}
and 
\begin{align}
\|\varphi_{k+1} - \varphi_k\|_{C^{2,\alpha}(M)} \leq \frac{1}{2} \|\varphi_k - \varphi_{k-1}\|_{C^{2,\alpha}(M)} \leq (\frac{1}{2})^k r^{\frac{1}{1-\alpha}} \delta_0.
\end{align}
Thus, we conclude that there exists some $\varphi_{\infty} \in \sH_{\omega}^{2,\alpha}$ such that $\varphi_k \rightarrow \varphi_{\infty}$ in $C^{2,\alpha}(M)$ as $k \rightarrow \infty$. Thus, we get
\begin{align}
F_r(\varphi_\infty) = 0.
\end{align}
From the regularity of elliptic equation, we could immediately see that $\varphi_\infty \in C^{\infty}(M)$. Also it's clear that 
\begin{align}\|\varphi_\infty\|_{C^{2,\alpha}(M)} \leq \|\varphi_1\|_{C^{2,\alpha}(M)} + r^{\frac{1}{1-\alpha}} \delta_0 \leq Cr \rightarrow 0, \text{ as } r \rightarrow 0.
\end{align}
Then we finish the proof of Theorem $\ref{thm1}$.

\end{proof}


\begin{thebibliography}{100}

\bibitem{A} T. Aubin, \emph{\'Equations du type Monge-Amp\`ere sur les vari\'et\'es k\"ahl\'eriennes compactes}. Bull. Sci. Math. (2) 102 (1978), no. 1, 63–95.

\bibitem{Ca1}E. Calabi. \emph{The space of K\"ahler metrics.} Proc. Int. Congr.
Math. Amsterdam 2, 206-207.

\bibitem{Ca2} E. Calabi. \emph{On K\"ahler manifolds with vanishing canonical class.} Algebraic geometry and topology. A symposium in honor of S. Lefschetz. Vol. 12. 1957.

\bibitem{Ca3} E. Calabi. \emph{Improper affine hyperspheres of convex type and a generalization of a theorem by K. J\"orgens}. Michigan Math. J. Volume 5, Issue 2 (1958), 105-126.

\bibitem{Ca4} E. Calabi.
\emph{Extremal K\"ahler metrics}. Seminar on Differential Geometry, volume 16 of 102, pages 259-290, Ann. of Math Studies, University Press, 1982. 

\bibitem{Ca5} E. Calabi. \emph{Extremal K\"ahler Metrics II}. Differential Geometry and Complex Analysis, pages 96-114, Springer, 1985.

\bibitem{C1}X. X. Chen. \emph{On the existence of constant scalar curvature K\" ahler metric: a new perspective.} arXiv:1506.06423.

\bibitem{CDS1} X. X. Chen, S. Donaldson and S. Sun. \emph{K\"ahler-Einstein metrics on Fano manifolds. I: Approximation of metrics with cone singularities}. J. Amer. Math. Soc. 28 (2015), pp. 183-197 (I).

\bibitem{CDS2} X. X. Chen, S. Donaldson and S. Sun. \emph{ K\"ahler-Einstein metrics on Fano manifolds. II: Limits with cone angle less than $2\pi$}. J. Amer. Math. Soc. 28 (2015), pp. 199-234.

\bibitem{CDS3} X. X. Chen, S. Donaldson and S. Sun. \emph{K\"ahler-Einstein metrics on Fano manifolds. III: Limits as cone angle approaches $2\pi$ and completion of the main proof}. J. Amer. Math. Soc. 28 (2015), pp. 235-278.

\bibitem{CPZ} X. X. Chen, Mihai P\u aun and Y. Zeng. \emph{On the deformation of extremal metrics.} arXiv:1506.01290v2 

\bibitem{D1} S. K. Donaldson. \emph{Conjectures in K\"ahler geometry.} Strings and geometry (2002): 71.

\bibitem{D2} S. K. Donaldson. \emph{Remarks on gauge theory, complex geometry and 4-manifold topology.} Fields Medallists’ Lectures, World Sci. Ser. 20th Century Math 5 (1997): 384-403.

\bibitem{D3} S. K. Donaldson, \emph{Scalar curvature and stability of toric varieties.} Jour.
Differential Geometry 62 (2002), 289–349.

\bibitem{D4} S. K. Donaldson, \emph{Interior estimates for solutions of Abreu’s equation.}
Collectanea Math. 56 (2005), 103–142.

\bibitem{D5}  S. K. Donaldson, \emph{Extremal metrics on toric surfaces: a continuity
method.} Jour. Differential Geometry 79 (2008), 384-432.

\bibitem{D6} S.K. Donaldson. \emph{Constant scalar curvature metrics on toric surfaces.} Geometric and Functional Analysis 19.1 (2009): 83-136.

\bibitem{JF} J. Fine. \emph{Constant scalar curvature K\"ahler metrics on fibred complex surfaces}. J. Differential Geom. Volume 68, Number 3(2004), 397-432.

\bibitem{GT} D. Gilbarg and N. S. Trudinger. \emph{Elliptic partial differential equations of second order.} springer, 2015.

\bibitem{MS} M. Lejmi and G. Sz\'ekelyhidi. \emph{The J-flow and stability}. \url{http://arxiv.org/abs/1309.2821v1}.

\bibitem{JS} J. Stoppa. \emph{Twisted constant scalar curvature K\"ahler metrics and K\"ahler slope stability}. 
J. Differential Geom. Volume 83, Number 3(2009), 663-691.

\bibitem{T1} G. Tian. \emph{K\"ahler-Einstein metrics with positive scalar curvature.} Inventiones Mathematicae 130.1 (1997): 1-37.

\bibitem{T2} G. Tian. \emph{Canonical metrics in K\"ahler geometry.} Birkhäuser, 2012.

\bibitem{RT} R. Thomas. \emph{Notes on GIT and symplectic reduction for bundles and varieties.} arXiv preprint math/0512411.

\bibitem{Y} S. T. Yau, \emph{On the Ricci curvature of a compact K\"ahler manifold and the complex Monge-Amp\`ere equation. I}. Comm. Pure Appl. Math. 31 (1978), no. 3, 339–411. 

\end{thebibliography}
\end{document}